\documentclass[12pt]{amsart}
\usepackage{amsmath,amsthm,amssymb,amsfonts}
\usepackage[active]{srcltx}
\usepackage{latexsym}
\usepackage{enumerate}
\usepackage{verbatim}
\usepackage{multicol}
\usepackage{graphicx}
\usepackage{epsfig}
\usepackage{color}
\usepackage{bezier}
\usepackage{curves}
\usepackage{fullpage}

\newtheorem{theorem}{Theorem}
\newtheorem{claim}{Claim}

\newtheorem{problem}{Problem}
\newtheorem{remark}{Remark}
\newtheorem{lemma}{Lemma}

\newtheorem{corollary}{Corollary}

\newcommand{\NN}{{\mathbb N}}
\newcommand{\ZZ}{{\mathbb Z}}

\newcommand*{\ackname}{{\bf Acknowledgements. }}

\title{Pairs of orthogonal countable ordinals}
\author[C.Laflamme] {Claude Laflamme}
\thanks{*The first author was supported by NSERC of Canada Grant \# 690404}
\address{University of Calgary, Department of Mathematics and Statistics, Calgary, Alberta, Canada T2N 1N4}
\email{laf@math.ucalgary.ca}
\author [M.Pouzet]{Maurice Pouzet}
\thanks{**This research was completed  while the second  author was visiting the Dept of Mathematics \& Computer Science, of the Royal Military College of Canada, in February 2012}
\address{ICJ, Math\'ematiques, Universit\'e
Claude-Bernard Lyon1, 43 Bd. 11 Novembre 1918
F$69622$ Villeurbanne  cedex, France and University of Calgary, Department of Mathematics and Statistics, Calgary, Alberta, Canada T2N 1N4} \email{pouzet@univ-lyon1.fr }
\author [N.Sauer]{Norbert Sauer}
\thanks{***The third author was supported by NSERC of
Canada Grant \# 691325} \address{University of Calgary, Department of
Mathematics and Statistics, Calgary, Alberta, Canada T2N 1N4}
\email{nsauer@math.ucalgary.ca}
\author[I.Zaguia]{Imed Zaguia}\thanks{****Corresponding author}
\address{Department of Mathematics \& Computer Science,  Royal Military College,
P.O.Box 17000, Station Forces, Kingston, Ontario, Canada K7K 7B4}
\email{zaguia@rmc.ca}

\date{\today}

\begin{document}

\begin{abstract} We characterize pairs of orthogonal countable ordinals.
Two ordinals $\alpha$ and $\beta$ are orthogonal if there are two linear orders
$A$ and $B$ on the same set $V$ with order types $\alpha$ and $\beta$ respectively such that
the only maps preserving both orders are the constant maps
and the identity map. We prove that if $\alpha$ and $\beta$ are two countable ordinals, with $\alpha \leq \beta$,
then $\alpha$ and $\beta$ are orthogonal if and only if either $\omega + 1\leq \alpha$ or $\alpha =\omega$ and $\beta < \omega \beta$.
\end{abstract}
\maketitle

\input{epsf}

\section{Introduction}

The following notion has been introduced by Demetrovics, Miyakawa, Rosenberg, Simovici and Stojmenovi\'{c}
\cite{dmrs}:

Two orders $P$ and $Q$ on the same set are \textit{orthogonal} if
their only common order preserving maps are the identity map and the
constant maps.

In this paper, we say that two ordinals $\alpha$ and $\beta$ are \emph{orthogonal} if there exist two linear orders
$A$ and $B$ on the same set $V$ with order types $\alpha$ and $\beta$ respectively such that
the only maps preserving both orders are the constant maps
and the identity map. Let $\omega$ be the first infinite  ordinal.

We prove:
\begin{theorem} \label{cio}
If $\alpha$ and $\beta$ are two countable ordinals, with $\alpha \leq \beta$,
then $\alpha$ and $\beta$ are orthogonal if and only if either $\omega + 1\leq \alpha$ or $\alpha =\omega$ and $\beta <\omega \beta$.
\end{theorem}

The proof of Theorem \ref{cio} will be done in Section \ref{section:
cio}, using the following result.

\begin{theorem}\label{omega}
There are  $2^{\aleph_0}$ linear orders $L$ of order type $\omega$ on $\NN$ such that $L$ is orthogonal to the natural order on $\NN$.\end{theorem}

This result will be given in Section \ref{section: omega}, and follows
from a simple construction which gives a bit more (see Corollary \ref{pz4} ).

Let us say few words about the history of this notion of orthogonality and our motivation.

The notion of orthogonality originates in the theory of clones.
The first examples of pairs of orthogonal finite orders were given by Demetrovics, Miyakawa, Rosenberg, Simovici and Stojmenovi\'c \cite{dmrs}; those
orders were in fact bipartite. More examples can be found in \cite{dr}. Nozaki, Miyakawa, Pogosyan
and Rosenberg \cite{rms} investigated the existence of a linear order orthogonal to a given finite linear order. They observed that there is always one provided that  the number of elements is  not equal to three and proved:

\begin{theorem} \cite{rms} \label{rms}
The proportion  $q(n)/n!$ of linear orders orthogonal to the natural order on $[n]:=\{1,...,n\}$ goes to $e^{-2}=0.1353...$ when $n$ goes to infinity.
\end{theorem}
Their counting argument was based on the fact that two linear orders
on the same \emph{finite} set are orthogonal if and only if they do
not have a common nontrivial interval. The notion capturing the
properties of intervals of a linear order was extended long ago to
posets, graphs and binary structures and a decomposition theory has
been developed (eg see \cite{fraisse2}, \cite{gallai},
\cite{fraisse3}, \cite{Eh-Ha-Ro}).  One of the terms in use for this
notion is \emph{autonomous set}; structures with no nontrivial
autonomous subset -the building blocks in the decomposition theory-
are called \emph{prime} (or \emph{indecomposable}). With this
terminology, the above fact can be expressed by saying that two linear
orders $\mathcal L$ and $\mathcal M$ on the same finite set $V$ are
orthogonal if and only if the binary structure $B: =(V, \mathcal L,
\mathcal M )$, that we call a \emph{bichain}, is prime. This leads to
results relating primality and orthogonality (\cite{Ri-Za}, \cite
{za}).

The notion of primality  has reappeared in recent years under a quite different setting:  a  study of permutations motivated by the Stanley-Wilf conjecture, now settled by Marcus and Tard\"os   \cite{Marcus}. This study,  which developed in many papers,  can be presented as follows:
To a permutation $\sigma$ on $[n]$ associate first  the linear order $\leq_{\sigma}$ defined by $x\leq_{\sigma}y$ if $\sigma(x)\leq \sigma(y)$ for the natural order on $[n]$;  next  associate the bichain $B_{\sigma}: =([n], (\leq, \leq_{\sigma}))$. On the set  $\mathfrak S:= \cup_{n\in \NN}\mathfrak S_n$ of all permutations,   set  $\sigma\leq \tau$ if $B_{\sigma}$ is embeddable into $B_{\tau}$.  Say that a subset $\mathfrak C$ of $\mathfrak S$ is \emph{hereditary} if $\sigma\leq \tau$ and $\tau\in \mathfrak C$ imply $\sigma\in \mathfrak C$.  The goal is to evaluate the growth rate of the function  $\varphi_{\mathfrak C}$ which counts for each integer $n$ the numbers $\varphi_{\mathfrak C}(n)$ of permutations $\sigma$ on $[n]$ which belong to $\mathfrak C$ (the Stanley-Wilf conjecture asserted  that $\varphi_{\mathfrak C}$ is bounded by an exponential if $\mathfrak C\not = \mathfrak S$).  For this purpose, simple permutations were introduced. A permutation $\sigma$ is \emph{simple} if $\leq_{\sigma}$ and the natural order $\leq$ on $[n]$ have no nontrivial interval in common. Arbitrary permutations being  obtained by means of simple permutations,  the enumeration of permutations belonging to a hereditary class of permutations can be then reduced to the enumeration of  simple permutations belonging to that class. This fact was illustrated in many papers (\cite{A-A},  \cite{klazar}, see also \cite{Br} for a survey on simple permutations and \cite{A-A-K}, where the asymptotic result mentioned in Theorem \ref{rms} is rediscovered). Notably, Albert and Atkinson \cite{A-A} proved that the generating series $\sum_{n\in \NN}\varphi_{\mathfrak C}(n)z^n$ is algebraic provided that $\mathfrak C$ contains only finitely many simple permutations. They asked for possible extensions of their result to hereditary sets  containing infinitely many simple permutations.

Tools of the theory of relations provide easy ways to produce examples
of hereditary sets containing infinitely many simple permutations (but
not to answer the Albert-Atkinson question).  Let us say that a subset
$\mathfrak C$ of $\mathfrak S$ is an \emph{ideal} if it is non-empty,
hereditary and \emph{up-directed}, this last condition meaning that
every pair $\sigma, \sigma' \in \mathfrak C$ has an upper bound
$\tau \in \mathfrak C$. Let us call \emph{age} of a bichain $B$ the
set $age (B):= \{\sigma \in \mathfrak S: B_{\sigma}\; \text{is embeddable into }\; B\}$.
Then, a subset $\mathfrak C$ of $\mathfrak
S$ is an ideal if and only if $\mathfrak C$ is the age of some bichain
$B$.  Furthermore, an ideal $\mathfrak C$ is the age of a prime
bichain if and only if every permutation belonging to $\mathfrak C$ is
dominated by some simple permutation belonging to $\mathfrak C$ (these
statements, which hold in the more general context of the theory of
relations, are respectively due to Fra\"{\i}ss\'e \cite{fraisse} and
Ille \cite{Ille}).  Because of these results, the study of ideals
leads to the study of countable prime bichains. It is then natural to
ask which are the possible pairs of order types of linear orders with this
property. Now, it must be noticed that in the infinite case, primality
and orthogonality no longer coincide.  Thus, the next question is about
pairs of orthogonal linear orders.  In \cite{sz} it was proved that
the chain of the rational numbers admits an orthogonal linear order of
the same order type. Here we examine the case of countable well
ordered chains. 

\section{Basic notations and results}

Let $V$ be a set. A \emph{binary relation} on $V$ is a subset $\rho$
of the cartesian product $V\times V$, but for convenience we write $x\rho y$ instead of
$(x,y)\in \rho$. A map $f:V\rightarrow V$ \emph{preserves} $\rho$ if:
$$x\rho y \Rightarrow f(x)\rho f(y)$$ for all $x,y \in V$.

These two notions are enough to present our results. In order to prove
them, we will need a bit more.

A \emph{binary structure} is a pair $R:=(V,(\rho_i)_{i\in I})$ where
$V$ is a set and each $\rho_i$ is a binary relation on $V$. If $F$ is
a subset of $V$, the restriction of $R$ to $F$ is $R_{\restriction F}:
=(F,((F\times F)\cap \rho_i)_{i\in I})$. If $R:=(V,(\rho_i)_{i\in I})$
and $R':=(V',(\rho^{\prime}_i)_{i\in I})$ are two binary structures, a
\emph{homomorphism} of $R$ into $R'$ is a map $f: V\rightarrow V'$
such that the implication

\begin{equation}\label{equ:homo}
x\rho_i y \Rightarrow f(x)\rho^{\prime}_i f(y)
 \end{equation}
holds for every $x,y\in V$, $i\in I$. If $f$ is one-to-one and
implication (\ref{equ:homo}) above is a logical equivalence, this is
an \emph{embedding}. If $R=R'$, a homomorphism is an
\emph{endomorphism}.  We will denote by $R\leq R'$ the fact that there
is an embedding of $R$ into $R'$ and by $R\leq_{fin}R'$ the fact that
$R_{\restriction (V\setminus F)}\leq R'$ for some finite subset $F$ of
$V$.

Let $R:=(V,(\rho_i)_{i\in I})$ be a binary structure. A subset $A$ of $V$ is {\it autonomous} (other terms are   interval,  module and clan)
in $R$ if for all $v\not\in A$ and for all $a,a^{\prime} \in A$ and for all $i\in I$, the following property holds:
\begin{equation*}
(v\rho_i a\Rightarrow v \rho_i a^{\prime})\;\mathrm{and}\;(a\rho_i v\Rightarrow a^{\prime} \rho_i v).
\end{equation*}

The empty set, the whole set $V$ and the singletons in $V$ are
autonomous sets and are called \textit{trivial}. We say that $R$ is
\emph{prime} if it has no nontrivial autonomous set, it is
\emph{semirigid} if the identity map and the constant maps are the
only endomorphisms of $R$ and it is \emph{embedding rigid} if the
identity map is the only embedding from $R$ to $R$. Finally, we say
that two binary relations $\rho$ and $\rho'$ on a set $V$ are
\emph{orthogonal} (or \emph{perpendicular}) if the
binary structure $(V, \rho, \rho')$ is semirigid.

\begin{remark}\label{rem:semi-prime}
Semirigidity is often defined for structures made of reflexive
relations. Indeed, in that case, all constant functions are
endomorphisms and, more generally, any map mapping an autonomous set
$A$ on an element $a\in A$ and leaving fixed the complement of $A$ is
an endomorphism. Thus, if $R$ is semirigid, $R$ must be prime, and in
any case, embedding rigid.
\end{remark}

The binary structures we consider in this paper are made of one or two
binary relations. Except in Section \ref{section: omega}, these binary
relations are orders, that is reflexive, antisymmetric and transitive
relations.  Our terminology on posets and chains agrees essentially
with \cite{Ro}.  An \emph{ordered set}, \emph{poset} for short, is a
pair $P:=(V, \mathcal \rho)$ where $\rho$ is an order relation on $V$.
Instead of $\rho$ and $(x,y)\in \rho$ we may use the symbol $\leq$ and
write $x\leq y$. When needed, we will use other symbols like $\leq'$,
$\leq_P$, $\mathcal P$.  The \emph{dual} of $P$ is $P^*:=(V,
\rho^{-1})$ where $\rho^{-1}:=\{(x,y):(y,x)\in \rho\}$. If $P:=(V,
\leq)$ is a poset, the \emph{comparability graph} of $P$, which we
denote by $Comp(P)$, is the graph whose vertex set is $V$ and whose edges are
the pairs $\{x,y\}$ such that either $x<y$ or $y<x$.  A \emph{chain}
is a poset in which the order is \emph{linear} (or total), a
\emph{bichain} is a binary structure made of two linear orders, say
$\mathcal L, \mathcal M$, on the same set $V$ that we will denote
$B:=(V, \mathcal L, \mathcal M)$ (instead of $B:=(V, (\mathcal L,
\mathcal M))$). If $C:=(V, \mathcal L)$ is a chain, an autonomous
subset is simply an interval, that is, a subset $A$ such that $x\leq
z\leq y$ and $x,y\in A$ imply $z\in A$. If $B:=(V, \mathcal L,
\mathcal M)$ is a bichain, an autonomous subset is a common interval
of the chains $(V, \mathcal L)$ and $(V,\mathcal M)$.  We suppose the
reader is familiar with the notions of order types of chains and of well
ordered chains, alias \emph{ordinal numbers}, as explained in
\cite{Ro}. In the proof of Theorem \ref{cio2}, an ordinal is the set
of strictly smaller ordinals. We denote by $\omega$ the order type of
the chain of non-negative integers $(\NN, \leq)$ where $\leq$ is the
natural order, and by $\omega^*$ the order type of its reverse. We use
Greek letters, like $\alpha, \beta, \dots$ to denote order types.  If
$\alpha$ and $\beta$ are order types of chains $C$ and $C'$, then we
set $\alpha \leq \beta$ if $C$ is embeddable into $C'$.  We denote by
$\omega \alpha$ the lexicographical sum of copies of $\omega$ indexed
by a chain of type $\alpha$. We note that for us $0$ is a limit
ordinal.

We also note that the condition $\beta <\omega \beta$ in Theorem
\ref{cio} could be expressed as $ind(\beta)< \omega^\omega$. Indeed,
define the \emph{right indecomposable part} of an ordinal $\gamma$ as
the ordinal $ind(\gamma)$ defined as follows: $ind(\gamma):=0$ if
$\gamma=0$ and otherwise $ind(\gamma):=\delta$ where $\delta$ is the least
nonzero ordinal such that $\gamma = \gamma' +\delta$ for some
$\gamma'$. Then observe that:

\noindent {\bf Fact.}
 For a nonzero ordinal $\gamma$ the following properties are equivalent.\\
(i) $ind(\gamma)\geq \omega^\omega$.\\
(ii) $\omega \gamma=\gamma$.\\

\noindent \emph{Proof of the Fact}.
According to the Cantor Decomposition Theorem, $\gamma
=\omega^{\beta_0}+\cdots+\omega^{\beta_k}$ with $\beta_0\geq \cdots
\geq \beta_k$. Hence, $\omega \gamma=
\omega^{1+\beta_0}+\cdots+ \omega^{1+\beta_k}$. Since
$ind(\gamma)=\omega^{\beta_k}$, $ind(\gamma)\geq \omega^\omega$ if and only if $\beta_k\geq \omega$. If $\beta_k \geq \omega$, then $\beta_i \geq \omega$ for every $i$, hence
$1+\beta_i = \beta_i$ for all $i$ and therefore $\omega \gamma=\gamma$.  Conversely, if $\omega \gamma=\gamma$, then since the Cantor decomposition is unique, $1+\beta_i=\beta_i$ for every $i$, amounting to $\beta_k\geq \omega$.\hfill$\Box$\\

Finally we recall that a map $f$ from a chain
$C:=(V, \leq)$ into itself is \emph{extensive} if $x\leq f(x)$ for
every $x\in V$. We will use several times the fact that a one-to-one
order preserving map on a well ordered chain is extensive.

We will use the following results.

\begin{theorem}\label{thm:gallai-kelly}
A poset $P:=(V, \mathcal P)$ is prime if and only if
$Comp(P)$ is prime.  Moreover, if $Comp(P)$ is prime then\\
(a) the edge set of $Comp(P)$ has exactly two transitive orientations  (namely $\mathcal P$ and $\mathcal P^{-1}$),\\
(b) if the order $\mathcal P$ is the intersection of two linear orders then no other pair of linear orders yields the
same intersection.
\end{theorem}

\begin{theorem}\label{zaguia}
Let $\mathcal L$ and $\mathcal M$ be two linear orders on the same set
$V$. Then the poset $P:=(V, \mathcal L\cap \mathcal M)$ is prime if
and only if the bichain $B:=( V, \mathcal L, \mathcal M)$ is prime.
\end{theorem}

Theorem \ref{thm:gallai-kelly} was obtained for finite posets in
\cite{gallai}, and extended to the infinite in \cite
{Ke}\label{thm:kelly}. Theorem \ref{zaguia} was stated in \cite{za}
for a finite sets; the proof given holds without that restriction.

A characterization of pairs of infinite orthogonal linear orders is
easy to state but, contrary to the finite case, it says nothing
about existence. Indeed:

\begin{theorem}\label{inf3}
Let $\mathcal L$ and $\mathcal M$ be two linear orders on the same set $V$. The following properties are equivalent.
\begin{enumerate}[(i)]
\item $\mathcal L$ and $\mathcal M$ are orthogonal.
\item The bichain $B:=(V,Ê\mathcal L,Ê\mathcal M)$ is prime and embedding rigid.
\item The poset $P:=(V,\mathcal L\cap \mathcal M)$ is prime and  has at most two embeddings: the identity map and some embedding of order $2$.
\end{enumerate}

\end{theorem}

\begin{proof} $(i)\Rightarrow (ii)$. This  follows from Remark \ref{rem:semi-prime}.\\
$(ii)\Rightarrow (iii)$. Since $B$ is prime, $P$ is prime by Theorem
\ref{zaguia}. Let $f$ be an embedding of $P$. Let $\mathcal
P:=\mathcal L\cap \mathcal M$, $\mathcal L_{f}:=\{(x,y) :
(f(x),f(y))\in \mathcal L\}$ and let $\mathcal M_{f}$ be defined
similarly.  Then $f$ is an embedding of $(V,\mathcal L_{f})$ into
$(V,\mathcal L)$ and an embedding of $(V,\mathcal M_{f})$ into
$(V,\mathcal M)$. Hence, $\mathcal P=\mathcal L_{f}\cap \mathcal
M_{f}$.  According to (b) of Theorem \ref{thm:gallai-kelly}, $\{\mathcal L,
\mathcal M\}=\{\mathcal L_{f}, \mathcal M_{f}\}$.\\
\noindent {\bf Case 1.}
$\mathcal L=\mathcal L_{f}$ and $\mathcal M= \mathcal M_{f}$. In this
case, since $f$ is an embedding of $(V,\mathcal L_{f})$ into
$(V,\mathcal L)$, it preserves $\mathcal L$; for the same reason, it
preserves $\mathcal M$. Thus $f$ is the identity.\\
\noindent{\bf Case 2.}
$\mathcal L=\mathcal M_{f}$ and $\mathcal M= \mathcal L_{f}$. In this
case, $f \circ f$ preserves necessarily $\mathcal L$ and $\mathcal M$
and therefore is an embedding of $B$. Thus $f \circ f$ is the
identity. \\

We check now that if $f$ is not the identity, this is the only
embedding of $P$ distinct from the identity. Indeed, let $f'$ be an
embedding of $P$ then $f'\circ f$ is an embedding too. If $f'$ is
distinct from the identity then $\mathcal L=\mathcal M_{f'}$ and
$\mathcal M= \mathcal L_{f'}$. It follows that $f'\circ f$ preserves
$\mathcal L$ and $\mathcal M$. Thus $f'\circ f$ is the identity,
amounting to $f'=f$.\\

\noindent $(iii)\Rightarrow (i)$.
Let $f$ be an endomorphism of $B$. If $f$ is not one-to-one, then the
inverse image of some element under $f$ is an interval for both chains
$(V, \mathcal L)$ and $(V, \mathcal M)$, hence it is an autonomous
subset in $B$. Since $P$ is prime, $B$ must be prime (the easy part of
Theorem \ref{zaguia}), hence the autonomous set is $V$ and therefore
$f$ is a constant map. If $f$ is one-to-one this is an embedding of
$B$, hence an embedding of $P$. If $f$ is not the identity, it must
have order $2$. But, since $f$ is an endomorphism of $B$, it preserves
$\mathcal L$ and since $\mathcal L$ is a linear order, the orbit of an
element $x$ not fixed by $f$ must be infinite. Hence, $f$ is the
identity.  \end{proof}

\section{Proof of Theorem \ref{omega}}\label {section: omega}

Let $A$ be a subset of $\NN$. Set $\hat A:=A\times \{1\}$ and
$\NN(A):=\NN\cup \hat A$. Let $G(A):=(\NN (A), E(A))$ be the graph
(undirected, with no loop) whose vertex set is $\NN(A)$ and edge set
$E(A):=\{\{n,n+1\} : n\in \NN\}\cup \{\{n,(n,1)\} : n\in A\}$.  For
instance, $G({\emptyset})$ is the infinite one way path; the graphs
$G({\{2\}})$ and $G({\NN})$ are depicted in Figure \ref{figure:1}.

Let $k\in \ZZ$, let  $t_k: \NN \longrightarrow \ZZ$ be the map defined by $t_k(n)=n+k$. We call \emph{translation} any map of that form.
\begin{lemma}\label{pz5} If  $1\not \in A$ then $G(A)$ is prime. If $A$ is not eventually periodic, $G(A)$ is embedding rigid.
\end{lemma}
\begin{proof}Suppose that   $1\not \in A$. Let $X$ be a nonempty autonomous set in $G(A)$ with more than one element. Then $X\cap \NN$ is  autonomous in $G(A)_{\restriction \NN}$. Since a path with at least four vertices is prime, $X\cap \NN$ must be  a trivial subset of $\NN$. The case $X\cap \NN=\{n\}$ is impossible. Indeed, since $X$ has at least two elements, it contains an element of  the form $(m,1)$. If $m\leq n$ or $n+2\leq m$, then $\{n+1, (m,1)\}$ is not an edge of $G(A)$, whereas $\{n,n+1\}$ is an edge. Since $X$ is autonomous, $n+1\in X$, which gives a contradiction. If $m=n+1$, then, since $1\not \in A$, $m\geq 2$ and thus $n-1$ is defined; since   $\{n-1, (m,1)\}$ is not an edge, whereas $\{n-1,n\}$ is an edge and $X$ is autonomous, $n-1\in X$. This  gives a contradiction. The case  $X\cap \NN=\emptyset$ is also impossible. Indeed, $X$  contains two elements of the  form $(m,1), (m',1)$. Since $\{m, (m+1)\}$ is an edge, whereas $\{m, (m'+1)\}$ is not, $m\in X$. It follows that $\NN\subseteq X$.  Let $m\in A$; since $\{m,(m,1)\}$ is an edge and $\{m+1, (m,1)\}$ is not, $(m,1)\in X$. Hence $X=N(A)$.   Thus $G(A)$ is prime.  Now,
Let $f$ be an embedding of $G(A)$ into $G(A)$.
This embedding maps vertices of degree $2$ to vertices of degree $2$ or $3$. Since those vertices belong to the infinite path $G(A)_{\restriction {\NN}}$, the map $f$ induces an embedding of $G(A)_{\restriction \NN\setminus \{0\}}$ into $G(A)_{\restriction {\NN}}$. This embedding is  a translation  $t_k$ for some  $k\in \NN$. The  elements of $A\setminus \{0\}$ are  the only elements of  degree $3$, from this we have $t_k(A)\subseteq A$. As is well known, if a translation sends a subset of $\NN$ into itself, this subset is eventually periodic.
\end{proof}

\begin{figure}[h]
\begin{center}
\leavevmode \epsfxsize=2in \epsfbox{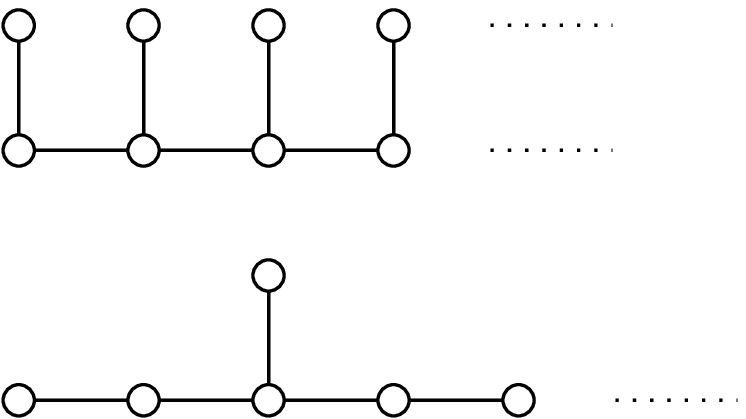}
\end{center}
\caption{} \label{figure:1}
\end{figure}

Let $A$ be  a subset of $ \NN$ and $P(A)$ be the poset whose vertex set is $\NN\cup \hat{A}$, the strict order being the union of the following two sets:

$$\{(2n,2n-1) : n>0\}\cup \{(2n,2n+1):n\geq 0)\}$$

and $$\{(2n,(2n,1)) : 2n\in A\}\cup \{((2n+1, 1), 2n+1):2n+1\in A\}.$$
\begin{lemma}\label{lem:PA}
The comparability graph of $P(A)$ is $G(A)$. The order on $P(A)$ is the intersection of a linear order $L(A)$ of order type $\omega$ and a linear order $L'(A)$ of order type $\omega^*$. If $1\not \in A$ then $P(A)$ is prime; if moreover $A$ is not eventually periodic, then $L(A)$ and $L'(A)$ are orthogonal.
\end{lemma}
\begin{proof}
 A representation of $P(A)$ into the  cartesian  product $\NN\times\NN$ is implicitly depicted in Figure \ref{fig:2}. The first component of the cartesian product is ordered with the natural order, the second is ordered by its reverse. The two lexicographical orders on the product yield two linear extensions of $P(A)$  of order type $\omega$ and $\omega^*$ respectively. Next, apply Lemma \ref{pz5}: If $1\not \in A$, then $G(A)$ is prime but then  $P(A)$ is prime. If $A$ is not eventually periodic,  then $G(A)$ is embedding rigid, but then,  trivially, $P(A)$ is embedding rigid. If both conditions hold, then $L(A)$ and $L'(A)$ are orthogonal by Theorem \ref{inf3}.
\end{proof}

\begin{figure}[h]
\begin{center}
\leavevmode \epsfxsize=2in \epsfbox{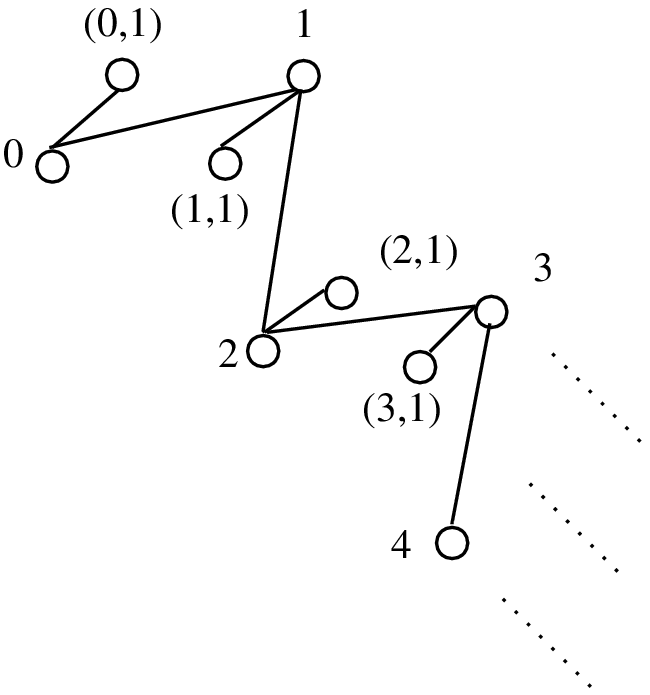}
\end{center}
\caption{} \label{fig:2}
\end{figure}

Let $A,A'$ be two subsets of a set $E$; we recall that $A$ is \emph{almost included} into $A'$, denoted $A \subseteq_{fin} A'$, if the set $A\setminus A'$ is finite.

The second part of Lemma \ref{pz5} extends as follows:

\begin{lemma}\label{pz1} Let $A, A'\in 2^{\NN}$. The following properties are equivalent:
\begin{enumerate}[(i)]
\item $G(A)$ is almost embeddable into $G(A')$.
\item There is some translation $t_k$ such that $t_k(A)$ is almost included in $A'$.
\end{enumerate}
\end{lemma}

\begin{proof}$(ii)\Rightarrow (i)$.
Let $t_k$  be a translation such that $t_k(A)$ is almost contained in $A'$. An integer $n\in \NN$ is \emph{bad} if either $t_k(n)\not \in \NN$ or $t_k(n)\not \in A'$.
 The set of bad integers is finite, hence the set $F$ of integers dominated by some bad  integer is finite too.
Let $X:=(\NN\cup \hat{A})\setminus (F\cup \hat {F}) $ and $\overline{t_k}: X \longrightarrow \NN \cup \hat{A'}$ defined by
$\overline{t_k}(n)=t_k(n)$ if $n\in \NN\setminus F$ and $\overline{t_k}(n,1)=(t_k(n),1)$ if $n\in A\setminus F$. This defines an embedding from $G({A})_{\restriction X}$ into $G({A'})$.\\
$(i)\Rightarrow (ii)$.
Let $f$ be an embedding of a restriction of $G(A)$ to a cofinite set of $\NN (A)$ into $G_{A'}$.
This embedding maps vertices of degree $2$ to vertices of degree $2$ or $3$. Since these vertices belong to the infinite path $G(A')_{\restriction {\NN}}$, the map $f$ induces an embedding of $G(A)_{\restriction Y}$, where $Y$ is an    infinite final segment of $\NN$,  into $G(A')_{\restriction {\NN}}$. Thus,
such a map is the restriction of some $t_k$ where $k\in \ZZ$. Since the  elements of $A$ and $A'$ (except at most one) are  the only elements of  degree $3$, $f$  maps almost all elements of $A$ into $ A'$, hence $t_k(A)$ is almost included in $A'$.
\end{proof}

\begin{lemma}\label{pz2} There is a family $\mathcal A$ of $2^{\aleph_0}$ subsets of $\NN$ such that for every pair $A, A'$ of distinct elements of $\mathcal A$, no translate of $A$ is almost included in $A'$.
\end{lemma}
\begin{proof} We present two proofs. For the first one,  start with $X:=\{x_n: n \in \NN\}$ where $x_0=0$  and $x_{n+1}=x_n + n$ (one just need the gaps increasing).
Now, let $\mathcal A$  be an almost disjoint family of $2^{\aleph_0}$ infinite subsets of  $X$.  For any $A\in \mathcal A$, and $n>0$, $A+n$ is almost disjoint from $X$, and thus almost disjoint from any other $A'$.

The second proof makes use of Sturmian words (\cite{Lo}, \cite{All-Sha}).
We identify subsets of $\mathbb{N}$ with their characteristic functions, that is  binary words. Hence, translating a set corresponds to shifting the word (note  that in this correspondence, if $u$ and $v$ are two infinite binary words,
we have $u \leq v$ iff $u^{-1}(1)\subseteq v^{-1}(1)$; however some translate of $u^{-1}(1)$ can be a contained into $v^{-1}(1)$ whereas no iterated shift of $u$ is almost contained in $v$, an observation leading to Problem
\ref{problem: shift} below).

Let $\alpha \in (0,1)\setminus \mathbb{Q}$. Let $X_\alpha$ be the set of Sturmian words whose slope is $\alpha$ (here
the \emph{slope} is the frequency of the letter $1$). The set $X_\alpha$ is a minimal uncountable subshift which is \emph{balanced}: for any two finite binary words $u$ and $v$ that appear as factors of elements of $X_\alpha$, if $u$ and $v$ have the same length, then $|u|_1-|v|_1\leq 1$ (where $|u|_1$ denotes the number of occurrences of the letter $1$ in $u$).

Let us define the following equivalence relation on $X_\alpha$: $x\sim y$ if there exists two integers $p$ and $q$ such that $S^p(x) = S^q(y)$ ($S$ denotes the shift map and $S^p$ its $p$ iterates, i.e $S$ is the map  from $2^{\NN}$ to $2^{\NN}$  defined by $S(x)_i:=x_{i+1}$ for $i\in \NN$. Since each class is countable, the quotient $X_\alpha/\sim$ is uncountable: Let $\mathcal{A}\subseteq X_\alpha$ be a system of representatives of $X_\alpha/\sim$.

Now, let $x$ and $y$ be two elements of $\mathcal{A}$ such that there exists an integer $n$ such that $S^n(y)$ is almost included in $x$, let us prove that $x=y$.

There exists an integer $k$ such that for any index $i$, $S^k(x)_i \geq S^{n+k}(y)_i$. If $S^k(x) = S^{n+k}(y)$, then $x=y$. Otherwise, there exists an index $i$ such that $S^k(x)_i =1$ and $S^{n+k}(y)_i=0$. Since $S^k(x)$ and $S^{n+k}(y)_i$ are elements of $X_\alpha$ which is balanced, $S^k(x)_i = S^{n+k}(y)_i$ for any $j\neq i$,
in particular $S^{k+i+1}(x) = S^{n+k+i+1}(y)$ and again $x=y$.
\end{proof}

 We thank Thierry Monteil \cite{monteil} for providing the second proof. We thank the referee for providing an alternative proof using words obtained by irrational rotations. For a link between irrational rotations and Sturm words, see Chapter 6 of \cite{fogg}.

\begin{problem}\label {problem: shift}Let us recall that a subset $X$ of $2^{\NN}$ is \emph{shift-invariant} if $S(X)\subseteq X$, where $S(X):=\{S(x):x\in X\}$. It is \emph{minimal} if it is non-empty,  compact,  shift-invariant and no proper  subset has the same properties (cf \cite{All-Sha}). For example the set $X_{\alpha}$ of Sturmian words with slope $\alpha$ is minimal. Is it true that every infinite minimal set $X$ contains a subset $X'$ of cardinality $2^{\aleph_0}$ such that  for every distinct $u,u'\in X'$, no translate of $u^{-1}(1)$ is almost contained in   $u'^{-1}(1)$?.
\end{problem}

\begin{corollary}\label{pz3}There is a family  $\mathcal A$ of subsets of $\NN\setminus\{0,1\}$ indexed by the positive reals  such that $G(A)$ is prime and embedding rigid for every $A\in \mathcal A$, and furthermore $G(A)\nleqslant_{fin} G(A')$ for
all distinct $A, A' \in \mathcal A$.
\end{corollary}
\begin{proof} Apply Lemma \ref{pz2}.
\end{proof}

 \begin{corollary}\label{pz4} There is a family  $\mathcal L$ of $2^{\aleph_0}$  linear orders  on $\NN$ of order type $\omega$ which are orthogonal
to $\omega$. Furthermore,  $(\NN, \leq,\leq_{L})\nleqslant_{fin} (\NN, \leq,\leq_{L'})$ for every distinct $L, L'\in \mathcal L$.
\end{corollary}
\begin{proof}
Select $\mathcal A$ as in Corollary \ref{pz3}  and apply Lemma \ref{lem:PA}. To  $A$  associate the prime poset $P(A)$  then the unique bichain  $B(A):=(\NN\cup \hat A, L(A), L'(A))$ where $L(A)$, $L'(A)$ have order type $\omega$ and $\omega^*$ respectively and  such that the intersection
$L(A)\cap L'(A)$ is the order of $P(A)$. The bichain $B(A)$ is semi-rigid and $B (A)\not\leq_{fin} B(A')$ for $A\not =A'$. Replace each  $B(A)$ by  the bichain  $B^{*}(A):=(\NN\cup \hat A, (L(A), L^{'*}(A))$ where $L^{'*}(A)$ is the dual of $L'(A)$. These bichains enjoy the same properties as the  $B(A)$'s. The components of these bichains being chains of order type $\omega$, we may suppose via a bijective map that their common domain is $\NN$ and the first order is the natural order.  This yields the above collection. \end{proof}


\section{Proof of Theorem \ref{cio}} \label{section: cio}

\begin{theorem}\label{cio1}
An ordinal $\alpha$ is orthogonal to $\omega$ if and only if $\alpha$
is countably infinite and $\omega \alpha >\alpha$.
\end{theorem}

Trivially, an ordinal $\alpha$ orthogonal to $\omega$ must be
countably infinite. The fact that $\omega \alpha >\alpha$ is a
consequence of the next lemma.

\begin{lemma}\label{lem:notrigid}
If $\alpha$ is a countably infinite order type (not necessarily an
ordinal) and $\omega \alpha \leq \alpha$, then the bichain $(\NN,
\leq, \leq_{L})$, where $L:=(\NN, \leq_{L})$ is a chain of order type
$\alpha$, is not embedding rigid.
\end{lemma}

For the proof of Lemma \ref{lem:notrigid}, we use the following
result, which is essentially Lemma 3.4.1 of \cite{mauricenejib}.

\begin{lemma}\label{lem:sierp}
Let $\alpha$ be a countably infinite order type,
$L_1:=(\NN,\leq_{L_1})$ and $L_2:=(\NN,\leq_{L_2})$ be two chains of
order type $\alpha$ and $\omega \alpha$ respectively. Then there is an
embedding of $(\NN, \leq, \leq_{L_1})$ into $(\NN, \leq,\leq_{L_2})$.
\end{lemma}

\noindent\emph{Proof of Lemma \ref{lem:notrigid}.}
Let $L_{1}:= (\NN, \leq_{L_1})$ be a chain with order type
$\alpha$. Since $\omega \alpha \leq \alpha$, there is a subset $X$ of
$\NN$, and in fact a proper subset, such that $L_{1\restriction X}$
has order type $\omega \alpha$. Let $h$ be the unique order
isomorphism from $(\NN, \leq)$ onto $(X, \leq_{\restriction X})$ and
let $L_2:=(\NN, \leq_{L_2})$ where $x\leq_{L_2} y$ amounts to
$h(x)\leq_{L_1}h(y)$. Clearly, $L_2$ has order type $\omega\alpha$ and
$h$ is an embedding of $(\NN, \leq, \leq_{L_2})$ into $(\NN,
\leq,\leq_{L_1})$. According to Lemma \ref{lem:sierp}, there is an
embedding, say $f$, of $(\NN, \leq, \leq_{L_1})$ into $(\NN,
\leq,\leq_{L_2})$. The map $h\circ f$ is an embedding of $(\NN, \leq,
\leq_{L_1})$ into $(\NN, \leq, \leq_{L_1})$. Since $X\not =\NN$, this
map is not surjective, hence this is not the identity and thus $(\NN,
\leq, \leq_{L_1})$ is not embedding rigid.\hfill $\Box$

It remains to prove that if $\alpha$ is countable and
$ind(\alpha)<\omega ^\omega$, then $\alpha$ is orthogonal to
$\omega$. Note that $0\neq ind(\alpha)<\omega ^\omega$ amounts to
$ind(\alpha)=\omega^n$ for some integer $n$. The proof will proceed by
induction on $n$ after some necessary lemmas.

\begin{lemma}\label{lem:embeddingchain}
Let $n<\omega$ and $f$ be an embedding from a chain $C$ of order type
$\omega^{n+1}$ into itself which is not the identity. Let
${(C_{\alpha})}_{\alpha<\omega^n}$ be the decomposition of $C$ into
intervals of order type $\omega$. Then there are $\alpha\leq \beta$ such
that $f$ is not the identity on $C_{\alpha}$ and $f(C_\alpha)\setminus
C_{\beta}$ is finite. Furthermore, $f(C_{\alpha}) \subseteq C_{\alpha}$ if $\alpha=\beta$.
\end{lemma}

\begin{proof}We mention at first that the existence of $\alpha$ and $\beta$ such that $f(C_{\alpha})\setminus C_{\beta}$ is finite implies $\alpha\leq \beta$ and $f(C_{\alpha}) \subseteq C_{\alpha}$ if $\alpha=\beta$. Indeed, since $C$ is well ordered, $f(x)\geq x$  for every $x \in C$, hence $C_{\alpha'}\cap f(C_{\alpha})=\emptyset$ for every $\alpha' < \alpha$. Now the proof of the lemma goes by induction on $n$. If $n=0$ then $C=C_0$. Set $\alpha=\beta=0$. Since
$f(C_0)\subseteq C_0$, we have $f(C_0)\setminus C_0= \emptyset$, thus
this set is finite, as required, and we are done.

Let $n\geq 1$ and suppose that the property holds for $n'<n$. Let
$(A_k)_{k<\omega}$ be the decomposition of $C$ into intervals of order type
$\omega^{n}$.

\begin{claim} \label{Claim1}
There is an embedding $\phi : \omega \rightarrow \omega$ such that for
each $k<\omega$, $f(A_k)\setminus A_{\phi(k)}$ has order type $<
\omega^{n}$.
\end{claim}

\noindent {\bf Proof of Claim \ref{Claim1}.}
Let $k<\omega$.  Set $i(k):=\{l<\omega: \; f(A_k)\cap A_l\neq
\emptyset \}$. This set is finite and nonempty.  Set $\phi(k):=\max
(i(k))$. As it is easy to see, the map $\phi$ is an embedding from
$\omega$ into $\omega$. \hfill $\Box$

\begin{claim} \label{Claim2}
There is some element $a\in A_0$ such that the decomposition of
$A'_{0}:=\{x : a\leq x\}\cap A_0$ into intervals of order type $\omega$ is
induced by the decomposition of $C$ into intervals of order type $\omega$
and $f(A'_0)\subsetneq A''_0:=A_{\phi(0)}$.
\end{claim}

\noindent {\bf Proof of Claim \ref{Claim2}.}
According to Claim \ref{Claim1}, $f(A_0)\setminus A''_{0}$ has order
type $< \omega^{n}$. Since $f(A_0)$ has order type $\omega^n$, there
is some $x\in A_0$ such that $f(\{y : x\leq y\}\cap A_0)\subseteq
A''_{0}$.  Pick $a>x$ in $A_0$ such that the decomposition of
$A'_{0}:=\{x: a\leq x\}\cap A_0$ into intervals of order type $\omega$ is
induced by the decomposition of $C$ into intervals of order type $\omega$.
\hfill $\Box$

With these two claims the proof of the lemma goes as follows. First,
with no loss of generality, we may suppose that $f$ is not the
identity on $A_0$. Otherwise, let $k_0$ be the least integer $k$ such
that $f$ is not the identity on $A_{k_0}$. Let $C':= C\setminus
\cup_{k<k_{0}} A_k$ and $(A'_k)_{k<\omega}$ be the decomposition of
$C'$ into intervals of order type $\omega^n$. Then $A'_0=A_{k_{0}}$ and $f$
induces an embedding $f'$ of $C'$ into itself which is not the
identity on $A'_0$. Thus, we may replace $C$ and $f$ by $C'$ and $f'$.

With our supposition, set $C':= C_{\restriction A'_{0}}$. Let $h$ be
an order isomorphism of $C_{\restriction A''_0}$ onto $C'$ and $f':=
h\circ f$. Then $f'$ is not the identity on $A'_0$. Induction applied to $C'$ and $f'$ yields some $\alpha'\leq
\beta'<\omega^{n-1}$ such that $f$ is not the identity on $C'_{\alpha'}$
and $f'(C'_{\alpha'})\setminus C'_{\beta'}$ is finite. Let $\alpha$ and
$\beta$ be such that $C'_{\alpha'}\subseteq C_{\alpha}$ and $C_{\beta}=h^{-1}(C'_{\beta'})$.
\end{proof}

\begin{lemma}\label{lem:10}
For every $n<\omega$ the  ordinals $\omega$ and $\omega^{n+1}$ are orthogonal.
\end{lemma}

\begin{proof}
If $n=0$, the result follows from Theorem \ref{omega}. We suppose that
$n\geq 1$.

\begin{claim}\label{Claim3}
There is a family of $(X_\alpha)_{\alpha<\omega^n}$ of infinite
subsets of $\NN$ such that none contains a nontrivial interval and
$0\not \in X_0$.
\end{claim}

\noindent {\bf Proof of Claim \ref{Claim3}.}
Let $(X_{\alpha})_{\alpha<\omega^{n}}$ be a  partition of $\NN$ into
infinitely many pairwise infinite subsets of $\NN$, where $X_0$ is the set of odd integers.\hfill $\Box$

\begin{claim}\label{Claim4}
There is a chain $C:=(\NN, \leq_C)$ of order type $\omega^{n+1}$ such
that $(X_{\alpha})_{\alpha <\omega^{n+1}}$ is the decomposition of $C$
into intervals of order type $\omega$.  If $C$ is such a chain then the
bichain $(\NN, \leq, \leq_C)$ is prime. Furthermore, if for each
$\alpha<\omega^n$, $B_\alpha :=(X_\alpha, \leq_{\restriction
X_{\alpha}}, {\leq_C}_{\restriction X_{\alpha}})$ is embedding rigid
and for $\alpha<\beta$, $B_\alpha \nleqslant_{fin} B_\beta$ then
$(\NN, \leq, \leq_C)$ is semirigid.
\end{claim}

\noindent {\bf Proof of Claim \ref{Claim4}.}
The existence of the chain $C$ is pretty obvious: on each set
$X_{\alpha}$ choose a linear order of order type $\omega$, and for $\alpha
<\beta$ put every element of $X_{\alpha}$ before every element of
$X_\beta$. Suppose for a contradiction that $(\NN, \leq, \leq_C)$
is not prime. Let $I$ be a nontrivial autonomous set, that is an
interval for $\leq$ and $\leq_C$. Due to our condition on the family
$(X_\alpha)_{\alpha <\omega^{n}}$, no $X_\alpha$ can contain $I$. Thus
there are $\alpha,\beta$ with $\alpha\neq \beta$ such that $X_\alpha$
and $X_\beta$ meet $I$. We may suppose that $\alpha<\beta$. In such a
case $I$ contains infinitely many elements of $X_\alpha$. This implies
that $I$ is a final segment of $(\NN, \leq)$. For each $n\in \NN$, let
$\phi(n)$ be such that $n\in X_{\phi(n)}$.  For each $i\not \in I$,
$X_{\phi(i)}\cap I$ is infinite. It follows that $\NN\setminus I$
cannot contain two distinct elements, hence $\NN \setminus
I=\{0\}$. Since $0\not \in X_0$, $0$ is not the least element of $C$,
hence $I=\NN\setminus \{0\}$ is not an interval of $C$, a
contradiction.

Finally, we show that $(\NN, \leq, \leq_C)$ is embedding rigid. Suppose for a contradiction that there is a proper embedding, say $f$. Since $f$ is an embedding of $C$, Lemma \ref {lem:embeddingchain}
ensures that there are $\alpha \leq \beta<\omega^{n}$ such that:

\begin{enumerate}
\item  $f$ is not the identity on $X_{\alpha}$.
\item  $f(X_{\alpha})$ is contained in $X_{\alpha}$ if $\beta=\alpha$ and  almost contained in $X_{\beta}$ otherwise.
\end{enumerate}

Since $B_\alpha$ is embedding rigid, $\alpha\neq \beta$. Set
$X'_{\alpha}=X_{\alpha}\cap f^{-1} (X_{\beta})$. Then $f_{\restriction
X'_{\alpha}}$ is an embedding of $B_{\alpha \restriction X'_{\alpha}}$
into $B_{\beta}$. Since $X_{\alpha}\setminus X'_\alpha$ is finite
$B_{\alpha}\leq_{fin} B_\beta$, a contradiction.\hfill $\Box$

According to Corollary \ref{pz4} in Section \ref{section: omega},
there is a family $(B_\alpha)_{\alpha<\omega^n}$ satisfying the
conditions of Claim \ref{Claim4}. Hence, $\omega$ and $\omega^{n+1}$
are orthogonal.
\end{proof}

Let $\alpha$ be an ordinal and let $0<n<\omega$. For $i<n$ define:
\[\overline{i}(mod\; n):=\{\beta<\alpha : \beta=\gamma +i+kn \mbox{ for some limit ordinal } \gamma \mbox{ and } k<\omega\}.\]

\begin{lemma}\label{equiclasses}Let $\alpha$ be an ordinal, $0<n<\omega$   and $i<n$.
\begin{enumerate}[(a)]
\item If $f$ is an embedding of $\alpha$ into itself which is the identity on
$\overline{i} (mod\; n)$, then $f$ is the identity map on $\alpha$.
\item If $\alpha$ is a limit ordinal and $I$ is an interval of $\alpha$ such that $\overline{i}(mod\;n)\subseteq I$, then $\alpha\setminus \{0, \dots, i-1\} \subseteq I$.
\end{enumerate}
\end{lemma}

\begin{proof}
$(a)$ We suppose that $f$ is not the identity and we argue for a contradiction. Let
$\beta\in \alpha$ such that $f(\beta)\not = \beta$. Set $\beta_0:=
\beta$ and $\beta_{m+1}:=f(\beta_m)$ for $m\in \NN$. Since $\alpha$ is
an ordinal and $f$ is an embedding, we have $\beta_0<\beta_1<\cdots
\beta_m<\cdots$. Hence, there is some $i'\in \overline {i} (mod \;n)$
and some $m\in \NN$ such that $\beta_m<i'<\beta_{m+1}$. Since $f$ is
order preserving $\beta_{m+1}=f(\beta_m)<f(i')$. Hence, $f$ is not the
identity on $\overline{i}(mod \; n)$.\\
$(b)$ Observe that $\alpha\setminus \{0, \dots, i-1\}$ is the least interval containing $\overline i (mod \;n)$.
\end{proof}

\begin{lemma}\label{arith1}Let $\alpha$ and $\beta$ be two infinite ordinals with $\alpha$ orthogonal to $\beta$ and $\gamma$ be an ordinal such that $|\gamma|\leq |\alpha|$, then
\begin{enumerate}[(i)]
\item $\alpha$ is orthogonal to $\gamma+\beta$.
\item $\alpha$ is orthogonal to $\beta+\gamma + 1$.
\end{enumerate}
\end{lemma}

\begin{proof}
Let $\alpha'$ be a limit ordinal and $n<\omega$ such that $\alpha= \alpha'+n$.
Let $C:= (V, \leq )$ be a chain with order type $\alpha$ and let $h$ be an order isomorphism from $\alpha$ onto $C$. Let $V'$ be the image of $\alpha'$ by $h$ and $U:=V\setminus V'$ (hence $\vert U\vert =n$).
Let
$\{V'_0,V'_1\}$ be the partition of $V'$, which is the image by $h$ of the partition of $\alpha'$ into
$\overline{0}(mod\;2)$ and $\overline{1}(mod \;2)$. Let $X\subseteq V'_1$ be such
that $X$ is an initial interval of $C\restriction V'_1$  and
$|X|=|\gamma|$ (this is possible since $\vert \gamma\vert  \leq \vert\alpha\vert$) and set $Y:=V\setminus X$. Observe that $C\restriction Y$ has order type $\alpha$.

(i) We notice  that if $\gamma<\omega$, then $\gamma+\beta$ is
isomorphic to $\beta$ and hence $\gamma+\beta$ is orthogonal to
$\alpha$. So we may assume that $\gamma$ is infinite.
We define  $C':=(V,\leq_{C'})$
such that  $C'\restriction X$ is an initial interval of $C'$ and has order type
$\gamma$, $Y$ is a
final interval of $C'$ with order type $\beta$ and $C'\restriction Y$ is orthogonal to
$C\restriction Y$. \\

By construction,  $C'$ has  order type $\gamma+\beta$.
We claim that $C'$ is orthogonal to $C$. We prove first that there is no proper common interval. Let $I$ be a
common interval  of $C$ and $C'$ such that   $|I|\geq 2$. We prove that $I=V$.
Let $J:= I\cap Y $. Then $J$ is an interval of $C\restriction Y$ and of $C'\restriction Y$. Since $C\restriction Y$ and $C'\restriction Y$ are orthogonal, either $J$ is empty, or $J=\{y\}$ for some $y \in Y$ or $J=Y$. In the first case, $I \subseteq X$, but  since, in $C$,  there is an element of $Y$ between any two elements of $X$, $I$ cannot be an interval of $C$. Thus this case is impossible. In the second case, there is some  $x\in I \cap X$. Since $\beta$ is a limit ordinal, the interval $[x,y]$ of $C'$ contains infinitely many elements of $X$, hence $I\cap X$ is infinite.  Since $X\subseteq V'_1$,  $I$ contains infinitely many elements of $Y$, a contradiction. Thus, we have $J=Y$. Since $V'$ is the unique  interval of $C$ containing $Y$,we have  $V' \subseteq I$ and since $Y\subseteq I$ we get $I=V$, as required.

 Let $f$ be an injective order preserving map
common to $C$ and $C'$. Then $f(Y)\subseteq Y$ (indeed, since  $C'$ is
 well ordered, $f$ is extensive on $C'$; since $Y$ is a final segment of $C'$, it follows that $f(Y)\subseteq Y$). Since $C'\restriction Y$ is orthogonal to
$C\restriction Y$, it
follows that  $f$  is the identity map on $Y$. From that, it follows that   $f(X)\subseteq X$ and also $f(V')\subseteq V'$. From Lemma \ref{equiclasses} (i) applied to $\alpha'$ and $f_{\restriction V'}$ it follows  that $f$ is the identity map on $V'$, thus $f$ is the identity map and we are done.\\

(ii) We consider two cases:\\

\noindent a) $\gamma$ is an infinite ordinal.

The proof follows the same lines as the proof of (i).
Let $u$ be the least element of $C$ and $v$ be the least element of $X$ in $C$, hence $v$ is the successor of $u$ in $C$. We set $Y':=Y\setminus\{u\}$ and $X':= X \cup \{u\}$.
We define  $C':=(V,\leq_{C'})$
such that  $C'\restriction Y'$ is an initial interval of $C'$,  has order type
$\beta$, and is orthogonal to $C\restriction Y'$, and $X'$ is a
final interval of $C'$ with order type $\gamma+1$ and $u$ and $v$ are its least and  largest element.

By construction,  $C'$ has  order type $\beta+\gamma+1$.
We claim that $C'$ is orthogonal to $C$. As before, we prove first that $C$ and $C'$ have  no proper common interval. Let $I$ be a
common interval  to $C$ and $C'$ such that   $|I|\geq 2$. We prove that $I=V$.

We set  $J:= I\cap Y'$. Then $J$ is an interval of $C\restriction Y'$ and of $C'\restriction Y'$. Since $C\restriction Y'$ and $C'\restriction Y'$ are orthogonal, either $J$ is empty, or $J=\{y'\}$ for some $y' \in Y'$ or $J=Y'$. In the first case, $I \subseteq X'$, but  since, in $C$,  there is an element of $Y$ between any two elements of $X$, $I$ cannot be an interval of $C\restriction X$. Hence $I=\{u, v\}$, but since  $u$ and $v$ are the extreme elements of the infinite chain $C'\restriction X'$,  this is impossible. In the second case, there is some  $x'\in I \cap X'$. The interval  $[y',x']$ of $C'$ must contain  the least element of $C'\restriction X'$, that is $u$; having more than one element, $I$  must contain  $v$, hence it contains $X'$, thus $I\cap X$ is infinite.  Since $X\subseteq X'_1$,  $I$ contains infinitely many elements of $Y$, a contradiction. Thus, we have $J=Y'$. Since $V'\setminus \{u,v\}$ is the smallest  interval of $C$ containing $Y'$ and $Y'\subseteq I$, we have $V\setminus \{u,v\}\subseteq I$.  Since  $V\setminus \{v\}$ is not an interval of $C$ and  $V\setminus \{u\}$ is not an interval of $C'$ it follows that $I=V$ as required.

Let $f$ be an injective order preserving map for both $C$ and $C'$. Then $f(v)=v$ (this is
because $f$ is extensive on $C'$ and $v$ is the largest element of $C'$). Also,
$f(u)=u$ (this is because $f(v)=v$ and $v$ is the successor of $u$ in
$C$). Since every element of $Y'$ is below $u$ in $C'$ we infer that
$f(Y')\subseteq Y'$. Since $C\restriction Y'$ and $C'\restriction Y'$ are orthogonal, $f$ is the identity on $Y'$. Hence,  $f(X')\subseteq X'$ and, from the fact that $f$ fixes the first two element of $C$, it follows that $f(V')\subseteq V'$. From Lemma \ref{equiclasses} (i) applied to $\alpha'$ and $f_{\restriction V'}$ it follows that $f$ is the identity map on $V'$, thus $f$ is the identity and we are done.\\

\noindent b) $\gamma$ is finite.

In this case, it suffices to prove that the conclusion holds if $\gamma=0$. Indeed,  a straightforward  induction yields the general conclusion.

Let $\beta'$ be a limit ordinal and $m<\omega$ such that $\beta= \beta'+m$. Let $C':= (V, \leq')$ be a chain of type $\beta$ with $C'$ orthogonal to $C$. Let $U'$ be the final segment of $C'$ such that $\vert U'\vert = m$. Since $U'$ is finite, $V'\setminus U'$ contains two elements $a$ and $b$ which are consecutive in $C$. Add to $V$ an extra element $v$. Let $C'_1$ be the chain obtained by putting $v$ as its last element and $C_1$ be obtained by inserting  $v$ between $a$ and $b$. Trivially, $C'_1$ has order type $\beta+1$;  since $\{a,b\}\subseteq V'$,   $C_1$ has order type $\alpha$. We claim that $C_1$ and $C'_1$ are orthogonal.

We prove first that there is no proper common interval. Let $I$ be a
common interval  to $C_1$ and $C'_1$ such that   $|I|\geq 2$. We prove that $I=V\cup \{v\}$.
Let $J:= I\cap V$. Then $J$ is an interval of $C_1\restriction V$ and of $C'_1\restriction V$. Since $C_1\restriction V=C$ and $C'_1\restriction V=C'$, these chains are orthogonal, hence either $J$ is empty, or $J=\{w\}$ for some $w \in V$ or $J=V$. The first case is impossible since  $\vert I\vert \geq 2$. In the second case, $v\in I$. But then $I$,  as an interval of $C'_1$, contains $a$ or $b$. But then, as an interval of $C'_1$, $I$ contains either the interval $[a, v]$ or the interval $[b, v]$ of $C'_1$. Both intervals are infinite, hence $I$ is infinite, contradicting the fact that $J$ is a singleton. Thus this case is impossible.  It follows that $J= V$. Since in $C_1$, $v$ is between two elements of $I$,  $v\in I$, thus $I=V\cup \{v\}$ as required.

Let $f$ be an injective order preserving map common to $C_1$ and $C'_1$. Then $f(v)=v$ (this is
because $f$ is extensive on $C'_1$ and $v$ is the largest element of $C'_1$). Thus $f$ induces an order preserving map of $C$ and $C'$. Since these chains are perpendicular, $f$ is the identity on $V$. It follows that $f$ is the identity on $V\cup \{v\}$ as required. The proof of the lemma is now complete.
\end{proof}

\noindent\emph{Proof of} Theorem \ref{cio1}.

Let $\alpha$ be a countable ordinal such that
$ind(\alpha)<\omega^\omega.$ We have $ind(\alpha)=\omega^n$ with
$n<\omega$. If $n=0$ then $\alpha= \omega +\gamma +1$. According to
Theorem \ref{omega}, $\omega$ is orthogonal to itself. Hence from
$(ii)$ of Lemma \ref{arith1}, $\omega$ is orthogonal to $\omega +
\gamma +1 =\alpha$. If $n>0$ apply Lemma
\ref{lem:10} and $(i)$ of Lemma \ref{arith1}. \hfill $\Box$

\begin{theorem} \label{cio2}
If $\alpha$ and $\beta$ are two countable ordinals, with $\omega +
1\leq \alpha\leq \beta$, then $\alpha$ and $\beta$ are orthogonal.
\end{theorem}
\begin{proof}

The case  $\alpha = \beta = \omega + 1$ follows from the fact that $\omega$ is orthogonal to $\omega$ and Lemma \ref{arith1} (ii) applied twice. Thus we may suppose $\beta \geq \omega+2$.

Let $\alpha'$ and $\beta'$ be such that $\alpha= \omega+1+ \alpha'$ and $\beta = \omega + 1 + \beta'$. Let $V$ be a countably infinite set disjoint from $\omega+1$ and
 let $\leq_A$ and $\leq_B$ be two orthogonal linear orders on $V$ of order type $\omega$. Let $W_{\alpha}$ and  $W_{\beta}$ be two disjoint subsets of the set of odd integers with cardinality $\vert \alpha'\vert $ and $\vert \beta'\vert $ respectively. Set $V':=V\cup W_{\alpha}\cup W_{\beta}\cup \{\omega\}$. We define two orthogonal linear orders $\leq_{\alpha}$ and $\leq_{\beta}$ on $V'$ with order type $\alpha$ and $\beta$ respectively, as follows

Let $f_{\alpha}$ be the order isomorphism
from $\omega\setminus W_{\beta}$ onto  $(V,\leq_A)$, let $g_{\alpha}$ be any bijection from $\alpha\setminus (\omega +1)$ onto  $W_{\alpha}$ and $1_{\beta}$ be the identity map on $W_{\beta}\cup \{\omega\}$. Then $h_{\alpha}:= f_{\alpha} \cup g_{\alpha}\cup  1_{\beta}$ is a bijective map from $\alpha$ onto $V'$. The order $\leq_{\alpha}$ is the image by $h_{\alpha}$  of the order on $\alpha$, thus has the same order type.  We define $\leq_{\beta}$ similarly with  $f_{\beta}$, $g_{\beta}$ and $1_{\alpha}$.

To show that $\leq_{\alpha}$ and $\leq_{\beta}$ are orthogonal, we first show that  they have no nontrivial common interval. So
suppose for a contradiction that there is some nontrivial common interval $I$. If
$I$ were to meet $V$ in at least 2 places, then since $I \cap V$ is a common interval
of $\leq_A$ and $\leq_B$, which are orthogonal, we would have $V \subseteq I$;  since $W_{\beta}$ is included in the least interval of $\omega$ containing $h_{\alpha}^{-1}(I)$ we would have   $W_{\beta} \subseteq I$
and,  similarly,  $W_{\alpha} \subseteq I$ and this would imply $\omega \in I$, hence $I=V'$, contradicting the non triviality of $I$. If $I$ were to meet $V\cup \{\omega\}$ in two places,
then it would also have to meet $V$ in two places, and we have just shown that
this cannot happen. So $I$ must meet either $W_{\alpha}$ or $W_{\beta}$. If it meets $W_{\alpha}$ then
it must also meet $V$ since it is a nontrivial interval of $\leq_{\beta}$ and $W_{\alpha}$ does not
contain any two successive numbers. But then since it meets both $V$ and $W_{\alpha}$ and
is an interval of $\leq_{\alpha}$ it must contain at least two elements of $V$, which is the
desired contradiction. The case that $I$ meets $W_{\beta}$ leads to a contradiction in a
symmetric way.

Next, we must show that there is no nontrivial embedding $f $ of $(V', \leq_{\alpha},\leq_{\beta})$ into itself. Suppose for a contradiction that there is such an
embedding. It must preserve the final segment $\{\omega\} \cup W_{\alpha}$ of $\leq_{\alpha}$ and the final segment $\{\omega\} \cup W_{\beta}$ of $\leq_{\beta}$, and so must fix
$\omega$. Thus since it respects $\leq_{\alpha}$ it must map $V$ into $V \cup W_{\beta}$ and since it respects $\leq_{\beta}$ it must map
$V$ into $V \cup W_{\alpha}$: putting these two facts together, it must map $V$ into itself.
But then since $\leq_A$ and $\leq_B$ are orthogonal, the restriction of $f$ to $V$ must be the
identity map. But then since $f$ respects $\leq_{\beta}$ the restriction of $f$ to $W_{\alpha}$ must be the
identity map. Similarly,  the restriction of $f$ to $W_{\beta}$ must be the identity map. Thus $f$ itself must be the identity map.
\end{proof}

\ackname The authors are grateful to an anonymous referee whose comments have improved the presentation of the paper and for offering a nice argument reproduced above that simplified the proof of Theorem \ref{cio2} and also an alternative argument for the proof of Lemma \ref{pz2}.

\end{document}